\documentclass[12pt,letterpaper]{amsart}
\usepackage{palatino, euler, epic, eepic, amssymb, xypic, microtype, enumerate, stmaryrd, epsfig, epstopdf}
\input xy
\xyoption{all}

 \newlength{\baseunit}               
 \newcount{\numlines}                
\usepackage{amsmath, amssymb, amscd, verbatim, xspace,amsthm}
\usepackage{latexsym, epsfig, color}

\newcommand\isom{\cong}

\newcommand\Proj{\operatorname{Proj}}

\newcommand\bq{\begin{equation}}
\newcommand\eq{\end{equation}}

\newtheorem{proposition}{Proposition}[section]
\newtheorem{theorem}[proposition]{Theorem}

\newtheorem{example}[proposition]{Example}

\newtheorem{lemma}[proposition]{Lemma}

\theoremstyle{definition}

\theoremstyle{remark}
\newtheorem{remark}[proposition]{Remark}
\usepackage{url}

\numberwithin{equation}{section}

\newcommand{\cut}[1]{}

\newcommand\hidden[1]{}

\newcommand{\PP}{\mathbb{P}}
\newcommand{\QQ}{\mathbb{Q}}
\newcommand{\RR}{\mathbb{R}}

\newcommand{\dra}{\dashrightarrow}                                    %
\newcommand{\ZZ}{{\mathbb{Z}}}                                        %
\newcommand{\cO}{{\mathcal O}}                                        %
\newcommand{\Bl}{\operatorname{Bl}}                                   
\newcommand{\Cl}{\operatorname{Cl}}                                   
\title{Some non-finitely generated Cox rings}

\author{Jos\'e Luis Gonz\'alez and Kalle Karu}
\address{J.L. Gonz\'alez,  Department of Mathematics, University of British Columbia,
  Vancouver, BC V6T1Z2, CANADA  \newline \indent
K. Karu,
Department of Mathematics, University of British Columbia, 
  Vancouver, BC V6T1Z2, CANADA} 
\email{jgonza@math.ubc.ca, karu@math.ubc.ca}
\thanks{This research was funded by NSERC Discovery and Accelerator grants.}

\begin{document}

\begin{abstract}
We give a large family of weighted projective planes, blown up at a smooth point, that do not have finitely generated Cox rings. We then use the method of Castravet and Tevelev to prove that the moduli space $\overline{M}_{0,n}$ of stable $n$-pointed genus zero curves does not have a finitely generated Cox ring if $n$ is at least 13.
\end{abstract}
\maketitle
\setcounter{tocdepth}{1} 




\section{Introduction}

We work over an algebraically closed field $k$ of characteristic $0$. In their recent article \cite{CastravetTevelev}, Castravet and Tevelev proved that the moduli space $\overline{M}_{0,n}$ does not have a finitely generated Cox ring when $n\geq 134$. They reduced the non-finite generation problem from the case of moduli spaces to the case of weighted projective planes blown up at the identity element $t_0$ of the torus. Examples of such weighted projective planes have been studied previously by many algebraists, as the Cox rings of these blowups appear as symbolic algebras of monomial prime ideals. Goto, Nishida and Watanabe \cite{GNW} gave an infinite family of weighted projective planes $\PP(a,b,c)$, such that $\Bl_{t_0} \PP(a,b,c)$ does not have a finitely generated Cox ring. The smallest such example, $\PP(25,29,72)$, was used by Castravet and Tevelev to get the bound $n=134$.

We extend these results by giving a large family of weighted projective planes $\PP(a,b,c)$, such that the blowup $\Bl_{t_0}\PP(a,b,c)$ does not have a finitely generated Cox ring. This family includes all examples of Goto, Nishida and Watanabe, but also weighted projective planes with smaller numbers, such as $\PP(7,15,26)$, $\PP(7,22,17)$, $\PP(12, 13, 17)$ (Table~\ref{table.one} below lists more such examples). More generally, we study projective toric surfaces $X_\Delta$ of Picard number $1$, such that the blowup $\Bl_{t_0} X_\Delta$ does not have a finitely generated Cox ring.

Using the reduction method of Castravet and Tevelev we prove:

\begin{theorem} \label{thm-moduli}
The moduli space $\overline{M}_{0,n}$ does not have a finitely generated Cox ring when $n\geq 13$.
 \end{theorem}

In the terminology of \cite{HuKeel}, the theorem implies that $\overline{M}_{0,n}$ is not a Mori dream space when $n\geq 13$. 
In contrast, it is known from \cite{HuKeel} that the variety $\overline{M}_{0,n}$ has a finitely generated Cox ring when $n\leq 6$ (see \cite{Castravet} for explicit generators in the case of $\overline{M}_{0,6}$). 
Remarkably, all log-Fano varieties have finitely generated Cox rings by \cite{BCHM}, but even though $\overline{M}_{0,n}$ is log-Fano for $n\leq 6$ that is not the case for $n>7$.   
In this way, Hu and Keel's question in \cite{HuKeel} of whether the Cox ring of $\overline{M}_{0,n}$ is finitely generated now only remains unsettled for $7\leq n\leq 12$.   

Let us recall that for a normal $\QQ$-factorial projective variety $X$ with a finitely generated class group $\Cl(X)$, a Cox ring of $X$ is any multigraded algebra of the form
\[ 
R(X; D_1,\ldots,D_r) = \bigoplus_{(m_1,\ldots,m_r) \in \ZZ^{r}} H^0(X, \cO_X(m_1D_1+\cdots+m_rD_r)), 
\] 
where $D_1,\ldots,D_r$ are Weil divisors whose classes span $\Cl(X) \otimes \QQ$. 
The finite generation of a Cox ring of $X$ is equivalent to the finite generation of every Cox ring of $X$, and it has strong implications for the birational geometry of $X$ (see \cite{HuKeel}). In the language of \cite{HuKeel}, $X$ is a Mori dream space if and only if $X$ has a finitely generated Cox ring.  


To construct the toric varieties $X_\Delta$, we start with a triangle $\Delta \subset \RR^2$ as shown in Figure~\ref{fig-tri1}. The vertices of $\Delta$ have rational coordinates, $(0,0)$ is one vertex, and the point $(0,1)$ lies in the interior of the opposite side. Such a triangle is uniquely determined by the slopes of its sides $s_1<s_2<s_3$. 

\begin{figure}[ht] \label{fig-tri1}
\centerline{\psfig{figure=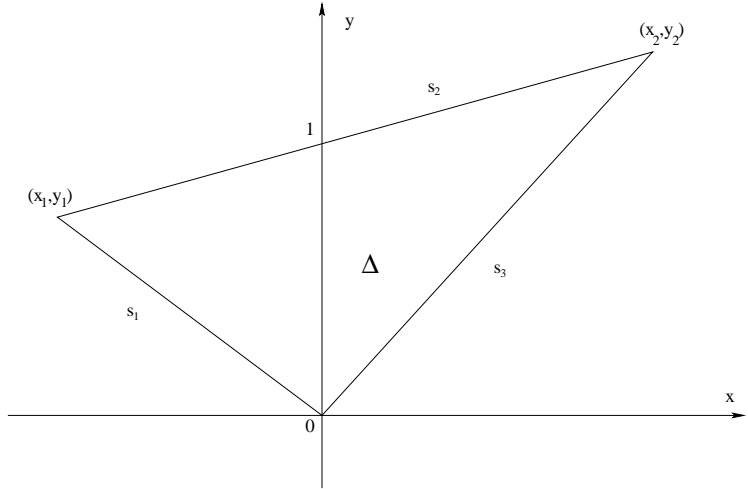,height=6.5cm}}
\caption{Triangle $\Delta$.}
\end{figure}

The triangle $\Delta$ defines a toric variety $X_\Delta$, whose fan is the normal fan of $\Delta$.  Let $\Bl_{t_0} X_\Delta$ be the blowup of $X_\Delta$ at the identity point of the torus $t_0 \in T\subset X_\Delta$.  

\begin{theorem} \label{thm-toric}
 Let the triangle $\Delta$ as in Figure~\ref{fig-tri1} be given by rational slopes $s_1<s_2<s_3$. The variety $\Bl_{t_0} X_\Delta$ does not have a finitely generated Cox ring if the following two conditions are satisfied:
\begin{enumerate}
 \item Let 
\[w = \frac{1}{s_2-s_1} + \frac{1}{s_3-s_2},\]
then $w<1$.
\item Let $n=|[s_1,s_2] \cap \ZZ|$. Then $|(n-1)[s_2, s_3]\cap \ZZ| = n$ and $n s_2 \notin \ZZ$.  
\end{enumerate}
\end{theorem}

The number $w$ in the theorem is the width of  $\Delta$: if $(x_1,y_1)$ and $(x_2,y_2)$ are the two nonzero vertices of $\Delta$, then $w=x_2-x_1$. To explain the second condition, consider a multiple $m\Delta$ that has integral vertices. A {\em column} in $m\Delta$ consists of all lattice points with a fixed first coordinate. Then $n$ is the number of lattice points in the second column from the left (i.e. with $x$-coordinate $m x_1+1$). The second condition  requires that the $n$-th column from the right (i.e. with $x$-coordinate $mx_2 - (n-1)$) contains exactly $n$ lattice points. Moreover, the $(n+1)$-th column from the right should not contain a lattice point on the top edge. 

It is  easy to construct examples of triangles $\Delta$ that satisfy the conditions of Theorem~\ref{thm-toric}. In fact, one can find a nonempty  open region in $\RR^3$, so that any rational point $(s_1,s_2,s_3)$ in that region defines such a triangle.

Different triangles may give rise to isomorphic toric varieties. However, if a triangle $\Delta$ exists with the property that $w<1$, then the two lattice points $(0,0)$ and $(0,1)$ in $\Delta$ are determined by the toric variety $X_\Delta$. 
The other triangles that give rise to toric varieties that are isomorphic by a toric morphism are obtained from $\Delta$ by applying an integral linear transformation that preserves the two lattice points. These transformations are generated by reflections across the $y$-axis and shear transformations $(x,y)\mapsto (x,y+ax)$ for $a\in\ZZ$. The shear transformation adds the integer $a$ to each of the three slopes and does not affect the two conditions of the theorem. The reflection replaces the slopes by their negatives and reflects the two conditions in a certain sense. (If $\Delta$ satisfies the two conditions, then its shear transform also satisfies them, but its reflection in general does not.)

\begin{example} Let    \label{example.wps}
\[ s_1=-\frac{2}{3},\quad s_2= \frac{1}{2}, \quad s_3 = 8.\]
 Then the two conditions of the theorem are satisfied with $w=104/105$ and $n=1$. The normal fan of $\Delta$ has rays generated by 
\[ v_1 = (2,3), \quad v_2 = (1,-2), \quad v_3 =(-8,1),\]
which satisfy the relation
\[ 15 v_1+26v_2+7v_3=0.\]
Moreover, since $v_1, v_2, v_3$ generate the lattice $\ZZ^2$, the toric variety $X_\Delta$ is the weighted projective plane $\PP(15,26,7)$. Then by Theorem~\ref{thm-toric}, $\Bl_{t_0} \PP(15,26,7) \cong \Bl_{t_0} \PP(26,15,7)$ does not have a finitely generated Cox ring. This last fact also admits a simpler direct proof, see Remark~\ref{simple.case}.
\end{example}

\begin{example} Take 
 \[ s_1 = -\frac{11}{3}, \quad s_2 = -\frac{4}{3}, \quad s_3 = \frac{2}{3}.\]
Then one checks that the two conditions are satisfied, with $w=3/7+3/6=13/14$, $n=2$. The normal fan of the triangle $\Delta$ has rays generated by
\[ v_1 = (11,3), \quad v_2 = (-4,-3), \quad v_3 =(-2,3),\]
satisfying the relation 
\[ 6v_1+13v_2+7v_3 = 0.\]
The vectors $v_1,v_2,v_3$ generate a sublattice of index $3$ in $\ZZ^2$. It follows that $X_\Delta$ is a quotient of $\PP(6,13,7)$ by an order $3$ subgroup of the torus. 
\end{example}

As the examples above illustrate, the toric varieties $X_\Delta$ are in general quotients of weighted projective planes $\PP(a,b,c)$ by a finite subgroup of the torus. We would like to know which weighted projective planes $\PP(a,b,c)$ correspond to triangles as in Theorem~\ref{thm-toric}. Let $e,f,g$ be positive integers, $\gcd(e,f,g)=1$, such that 
\[ ae+bf-cg = 0. \]
We call $(e,f,-g)$ a relation for $\PP(a,b,c)$.

\begin{theorem} \label{thm-proj}
Let  $\PP(a,b,c)$ be a weighted projective plane with relation $(e,f,-g)$. Then the blowup $\Bl_{t_0} \PP(a,b,c)$ does not have a finitely generated Cox ring if the following conditions are satisfied:
\begin{enumerate}
 \item Let 
\[ w= \frac{g^2c}{ab}.\]
Then $w<1$.
\item Let $n$ be the number of integers $\delta \leq 0$, such that
\[ (b,a)+\delta(e,-f)\]
has non-negative components, both divisible by $g$. Then there must exist exactly $n$ integers $\gamma\geq 0$, such that 
\[ (n-1)(b,a)+\gamma(e,-f)\]
has non-negative components, both divisible by $g$. Moreover, $n(b,a) \neq (0,0) \pmod{g}.$
\end{enumerate}
\end{theorem}

To find whether the theorem applies to a weighted projective plane $\PP(a,b,c)$, one has to consider all relations $(e,f,-g)$, possibly after permuting $a,b,c$. However, there can be at most one such relation satisfying $w<1$, even after permuting $a,b,c$.  To find this relation, one only needs to consider the values $g<\sqrt{ab/c}$ and search for $e$ and $f$. In any case, finding the relation and checking the two conditions of the theorem are best done on a computer. Using a computer we found $6814$ weighted projective planes $\PP(a,b,c)$ with $a,b,c\leq 100$ that satisfy the conditions of the theorem. The $42$ cases with  $a,b,c\leq 30$ are listed in Table~\ref{table.one}.

\begin{table}[ht]     \label{table.one}
\begin{minipage}[b]{0.3\linewidth}\centering
\begin{tabular}{| c | c |}
\hline
\hline 
$\PP(a,b,c)$ & $(e,f,-g)$\\ 
\hline
$\PP(7, 15, 26)$ & $(1, 3, -2)$ \\ 
$\PP(7, 17, 29)$ & $(1, 3, -2)$ \\ 
$\PP(7, 22, 17)$ & $(1, 2, -3)$ \\ 
$\PP(7, 25, 19)$ & $(1, 2, -3)$ \\ 
$\PP(10, 11, 27)$ & $(1, 4, -2)$ \\ 
$\PP(10, 21, 13)$ & $(1, 2, -4)$ \\ 
$\PP(10, 29, 17)$ & $(1, 2, -4)$ \\ 
$\PP(11, 21, 25)$ & $(3, 2, -3)$ \\ 
$\PP(12, 13, 17)$ & $(1, 3, -3)$ \\ 
$\PP(12, 19, 23)$ & $(1, 3, -3)$ \\ 
$\PP(12, 25, 29)$ & $(1, 3, -3)$ \\ 
$\PP(13, 9, 29)$ & $(1, 5, -2)$ \\ 
$\PP(13, 18, 25)$ & $(3, 2, -3)$ \\ 
$\PP(14, 29, 25)$ & $(3, 2, -4)$ \\ 
\hline
\end{tabular}
\end{minipage}
\hspace{0.5cm}
\begin{minipage}[b]{0.3\linewidth}
\centering
\begin{tabular}{|c|c|}
\hline
\hline 
$\PP(a,b,c)$ & $(e,f,-g)$\\ 
\hline
$\PP(16, 25, 11)$ & $(1, 2, -6)$ \\ 
$\PP(17, 13, 23)$ & $(1, 4, -3)$ \\ 
$\PP(17, 16, 27)$ & $(1, 4, -3)$ \\ 
$\PP(17, 21, 20)$ & $(1, 3, -4)$ \\ 
$\PP(17, 25, 23)$ & $(1, 3, -4)$ \\ 
$\PP(17, 29, 26)$ & $(1, 3, -4)$ \\ 
$\PP(18, 23, 25)$ & $(3, 2, -4)$ \\ 
$\PP(19, 11, 13)$ & $(1, 3, -4)$ \\ 
$\PP(19, 22, 26)$ & $(2, 3, -4)$ \\ 
$\PP(19, 26, 29)$ & $(2, 3, -4)$ \\ 
$\PP(19, 27, 20)$ & $(1, 3, -5)$ \\ 
$\PP(19, 29, 11)$ & $(1, 2, -7)$ \\ 
$\PP(20, 21, 26)$ & $(1, 4, -4)$ \\ 
$\PP(20, 22, 27)$ & $(1, 4, -4)$ \\ 
\hline
\end{tabular}
\end{minipage}
\hspace{0.5cm}
\begin{minipage}[b]{0.3\linewidth}
\centering
\begin{tabular}{|c|c|}
\hline
\hline 
$\PP(a,b,c)$ & $(e,f,-g)$\\ 
\hline
$\PP(22, 13, 29)$ & $(1, 5, -3)$ \\ 
$\PP(22, 21, 17)$ & $(1, 3, -5)$ \\ 
$\PP(23, 28, 25)$ & $(3, 2, -5)$ \\ 
$\PP(24, 13, 19)$ & $(1, 4, -4)$ \\ 
$\PP(24, 17, 23)$ & $(1, 4, -4)$ \\ 
$\PP(24, 26, 17)$ & $(1, 3, -6)$ \\ 
$\PP(26, 18, 29)$ & $(1, 5, -4)$ \\ 
$\PP(27, 10, 29)$ & $(1, 6, -3)$ \\ 
$\PP(27, 17, 28)$ & $(1, 5, -4)$ \\ 
$\PP(27, 19, 14)$ & $(1, 3, -6)$ \\ 
$\PP(27, 22, 23)$ & $(1, 4, -5)$ \\ 
$\PP(27, 25, 17)$ & $(1, 3, -6)$ \\ 
$\PP(29, 19, 21)$ & $(1, 4, -5)$ \\ 
$\PP(29, 30, 17)$ & $(1, 3, -7)$ \\
\hline 
\end{tabular}
\end{minipage}
\label{tab-30}
\\ [2ex]
\caption{Weighted projective planes $\PP(a,b,c)$, $a,b,c\leq 30$,  with relation $(e,f,-g)$, that satisfy the conditions of Theorem~\ref{thm-proj}. }

\end{table}

\begin{example}
 Consider $\PP(19,11,13)$, with relation $(e,f,-g)=(1,3,-4)$. We check that the two conditions hold:
\begin{enumerate}
 \item $w = 208/209 <1.$
\item The set of integers $\delta \leq 0$, such that 
\[ (11,19)+ \delta(1,-3) \]
has non-negative components divisible by $4$ is $\delta\in \{-3, -7, -11\}$. Hence $n=3$. Now the set of integers $\gamma\geq 0$, such that 
\[ 2(11,19)+ \gamma(1,-3) \]
has non-negative components divisible by $4$ is $\gamma\in\{ 2,6,10\}$. Finally, $3(11,19) \neq (0,0) \pmod{4}$. 
\end{enumerate}
It follows that $\Bl_{t_0} \PP(19,11,13)$ does not have a finitely generated Cox ring.
\end{example}

\begin{example} [Goto, Nishida, Watanabe \cite{GNW}] Consider the family of weighted projective planes $\PP(7N-3, 8N-3, (5N-2)N)$, where $N\geq 4$, $3 \nmid N$. We check that the two conditions are satisfied with relation $(e,f,-g) = (N, N, -3)$. Note that we need $3 \nmid N$ for $\gcd(e,f,g)=1$.
 \begin{enumerate}
\item \[ w= \frac{9(5N-2)N}{(7N-3)(8N-3)} <1, \quad \text{when $N\geq 3$}.\]
\item For $\delta\in \{-2, -5\}$, 
\[ (8N-3,7N-3)+\delta(N, -N)\]
has nonnegative components divisible by $3$. Hence $n=2$. For $\gamma\in \{1,4\}$, 
\[ (8N-3,7N-3)+\gamma(N, -N)\]
has nonnegative components divisible by $3$. Moreover, since $3 \nmid N$,
\[ 2(8N-3, 7N-3) \neq (0,0) \pmod{3}.\]
 \end{enumerate}
Similarly, the other family considered by Goto, Nishida and Watanabe in \cite{GNW}, $\PP(7N-10, 8N-3, 5N^2-7N+1)$, $N\geq 5$, satisfies the two conditions with relation $(e,f,-g)=(N, N-1, -3)$. (In fact, the case $N=3$ of this family, $\PP(11,21,25)$, is listed in Table~\ref{table.one}.)
\end{example}


\section{Proof of Theorem~\ref{thm-toric}}

We start the proof using a geometric argument as in \cite{CastravetTevelev}.

The toric variety $X_\Delta$ is $\QQ$-factorial and $\Cl(X_\Delta)\otimes \QQ$ is $1$-dimensional, with basis the class of the $\QQ$-divisor $H$ corresponding to the triangle $\Delta$. Denote $X=\Bl_{t_0} X_\Delta$. Then $\Cl(X) \otimes \QQ$ is $2$-dimensional, with basis the classes of the exceptional divisor $E$ and the pullback of $H$, which we also denote $H$. 

Recall from toric geometry \cite{Fulton} that lattice points in $m\Delta$ correspond to certain torus-invariant sections of $\cO_{X_\Delta}(mH)$. We identify a lattice point $(i,j)$ with the monomial $x^iy^j$ considered as a regular function on the torus $T$. The two lattice points $(0,0)$ and $(0,1)$ define a section $1-y$ of $\cO_{X_\Delta}(H)$. We let $C$ be the strict transform of this curve in $X$. Then $C$ has class $[H-E]$. Since the self-intersection $H^2$ is equal to twice the area of $\Delta$ (which equals $w$), we get
\[ C^2 = H^2+E^2 = w-1 <0.\] 
It follows that $C$ and $E$ are two negative curves on $X$ whose classes generate the Mori cone of curves of $X$ (which in this case coincides with the pseudoeffective cone of $X$). Its dual, the nef cone of $X$, is generated by the class of $H$ and the class $D=[H-wE] \in C^\perp$. 

By a result of Cutkosky \cite{Cutkosky}, the Cox ring of $X$ is finitely generated if and only if there exists an integer $m>0$, such that some effective divisor in the class $mD$ does not have the curve $C$ as a component. We will fix an $m$ large and divisible enough, such that $mD$ is integral and prove that any section of $\cO_X(mD)$ vanishes on $C$. We may replace $m$ by an integer multiple if necessary. Notice that this implies that although $D$ generates an extremal ray of the nef cone of $X$ it is not semiample, so the Cox ring of $X$ cannot be finitely generated by \cite[Definition 1.10 and Proposition 2.9]{HuKeel}.

A divisor in the class $mD$ is defined by a Laurent polynomial (considered as a regular function on the torus $T$) 
\[ f(x,y) = \sum_{(i,j)\in m\Delta \cap\ZZ^2} a_{ij} x^iy^j,\]
that vanishes to order at least $W=mw$ at the point $t_0=(1,1)$. In other words, all partial derivatives of $f$ of order up to $W-1$ vanish at $t_0$. Now it suffices to prove that for such $f$, the coefficient $a_{mx_1,my_1}$ at one of the nonzero vertices of $m\Delta$ is zero. Indeed, this implies that the section defined by $f$ vanishes at the $T$-fixed point in $X_\Delta$ corresponding to the vertex. Similarly, the curve $C$ passes through that fixed point, but since $C\cdot D=0$, it follows that $f$ must vanish on $C$.

\begin{remark} There is a more algebraic argument for the claim in the previous paragraph. We want to prove that $f(x,y)$ vanishes on $C$, in other words, that $1-y$ divides $f$. This happens if and only if the column sums
\[ c_i = \sum_{(i,j) \in m\Delta\cap \ZZ^2} a_{ij}\]
all vanish. There are $W+1$ column sums $c_i$. The derivatives $\partial_x^l$ for $l=0,\ldots,W-1$ give $W$ linearly independent relations on $c_i$. If we can find one more linearly independent relation, then $c_i=0$ for all $i$. The vanishing of $a_{mx_1,my_1}$ gives such an extra relation.
\end{remark}  

We first transform the triangle $m\Delta$ by integral translations and shear transformations $(i,j) \mapsto (i, j+ai)$ for $a\in\ZZ$. The translation operation multiplies $f$ with a monomial, and the shear transformation performs a change of variables on the torus. The two operations do not affect the order of vanishing of $f$ at $t_0$ or the conditions of the theorem. The shear transformation has the effect of adding the integer $a$ to each of the three slopes $s_1,s_2,s_3$. 

Let us start by bringing $m\Delta$ to the form shown in Figure~\ref{fig-tri2}. We first apply a shear transformation, so that $-2< s_2 < -1$. (Note that $s_2\notin\ZZ$ by condition (2).) We then translate the triangle so that $(mx_1,my_1)$ moves to a point with $x$-coordinate $-2$ and  $(mx_2,my_2)$ moves to a point on the $x$-axis. Call the transformed triangle $\tilde{\Delta}$. Note that the transformations do not change the number of lattice points in the columns. In particular, the second column from the left in $\tilde{\Delta}$ again contains $n$ lattice points.

\begin{figure}[ht] \label{fig-tri2}
\centerline{\psfig{figure=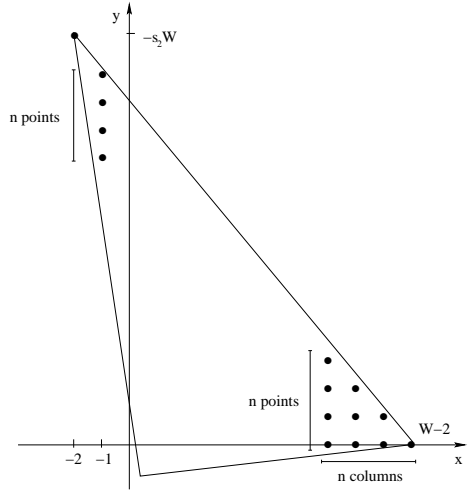,height=6.5cm}}
\caption{Triangle $\tilde{\Delta}$.}
\end{figure}

Consider now the right $n$ columns of the triangle $\tilde{\Delta}$. If $n=1$, then no more preparations are needed. For $n>1$, we may assume that the second column from the right contains at least $2$ lattice points. Otherwise apply the reflection $(i,j)\mapsto (-i,j)$ to the original triangle $\Delta$ to reduce to the case $n=1$. Since $-2\leq s_2 < -1$, it follows that the lattice points in the second column from the right have $y$-coordinates $1,0,\ldots.$ By condition (2), the $n$-th column from the right contains exactly $n$ lattice points, which must have $y$-coordinates $0,1,\ldots,n-1$. It now follows that for any $j=1,\ldots,n$, the $j$-th column from the right contains exactly $j$ lattice points with $y$-coordinates $0,1,\ldots, j-1$. In summary, the lattice points in the $n$ columns on the right are  
\[ (W-n-1+i, j), \quad i,j\geq 0, \quad i+j< n.\]

Consider a derivative 
\[ ( \sum_{i=0}^n \alpha_i \partial_x^i \partial_y^{n-i}) \partial_x^{W-n-1}\]
of order $W-1$. This derivative vanishes on all monomials $x^iy^j$ for $(i,j)\in\tilde{\Delta}\cap\ZZ^2$, except for the monomials with $i<0$. There are $n+1$ such monomials, corresponding to lattice points in the left two columns of $\tilde{\Delta}$. We claim that there exist coefficients $\alpha_i$, such that the derivative, when evaluated at $t_0$, vanishes on all $n$ monomials corresponding to lattice points in the second column, and it does not vanish on the monomial corresponding to the vertex. This implies that the coefficient  in $f$ of the monomial corresponding to the vertex must be zero.  

Consider  monomials $x^iy^j$ with $(i,j)\in \tilde{\Delta}\cap\ZZ^2$ and $i<0$. Let us first apply $\partial_x^{W-n-1}$ to these monomials. The result is the set of monomials (with nonzero coefficients that we may ignore)
\[ x^{-(a+1)}y^{b+n+1}, x^{-a}y^{b+j}, \quad j=0,\ldots,n-1.\]
Here $a=W-n$ and $b=-s_2W-n-1$. Making $m$ (and hence also $W$) bigger if necessary, we may assume that both $a,b>0$. Lemma~\ref{lem-deriv} below shows that we can choose the desired coefficients $\alpha_i$ if $a(n+1) \neq bn$. This condition is equivalent to $-s_2\neq 1+1/n$, which follows from the assumption that $ns_2\notin\ZZ$. \qed

\begin{lemma} \label{lem-deriv}
 Let $n,a,b>0$ be integers, such that $a(n+1)\neq bn$. Consider two sets of monomials:
\[ S_1= \{ x^{-(a+1)}y^{b+n+1}\}, \quad S_2=\{x^{-a}y^{b+j}\}_{j=0,\ldots,n-1}.\]
Then there exists a derivative 
\[D = \sum_{i=0}^n \alpha_i \partial_x^{n-i} \partial_y^{i},\]
such that $D$ applied to every monomial in $S_2$ vanishes at $t_0=(1,1)$, and $D$ applied to the monomial in $S_1$ does not vanish at $t_0$.
\end{lemma}

\begin{proof}
To start, there exists a nonzero derivative
 \[ \tilde{D} = \sum_{i=0}^n \beta_i \partial_y^i,\]
such that $\tilde{D}$ applied to monomials $y^{b+j}$ for $j=0,\ldots, n-1$, vanishes at $y=1$. By Lemma~\ref{lem-deriv1} below we may take
\[ \beta_i = (-1)^i \frac{(b+n-i-1)!}{(b-1)!} {n \choose i}.\]
Now let
\[ \alpha_i = (-1)^i \frac{\beta_i}{a\cdot(a+1)\cdots(a+n-i-1)} = \frac{[b+n-i-1]_{n-i}}{[a+n-i-1]_{n-i}} {n \choose i},\]
where we used the notation
\[ [k]_l = k(k-1)\cdots(k-l+1).\]
With these coefficients $\alpha_i$, the derivative $D$ applied to the monomials in $S_2$ vanishes at $t_0$. We need to prove that $D$ applied to the monomial in $S_1$ does not vanish at $t_0$.    
 
We apply $D$ to $x^{-(a+1)}y^{b+n+1}$ and evaluate at $t_0$ to get
\[ \sum_{i=0}^n  \frac{[b+n-i-1]_{n-i}}{[a+n-i-1]_{n-i}} {n \choose i} (-1)^{n-i} [a+n-i]_{n-i} [b+n+1]_i.\]
Now simplify:
\[ \frac{[a+n-i]_{n-i}}{[a+n-i-1]_{n-i}} = \frac{a+n-i}{a},\]
\[ [b+n-i-1]_{n-i}  [b+n+1]_i = \frac{[b+n+1]_{n+2}}{(b+n-i)(b+n-i+1)}.\]
Replacing $i$ by $n-i$, the sum becomes 
\[ \frac{[b+n+1]_{n+2}}{a} \sum_{i=0}^n (-1)^i\frac{a+i}{(b+i)(b+i+1)} {n\choose i}.\]
We can further express
\[  \frac{1}{(b+i)(b+i+1)} {n\choose i} = \frac{[b+i-1]_{b-1}}{[b+n+1]_{b+1}} {b+n+1 \choose b+i+1},\]
hence the sum is 
\[ \frac{[b+n+1]_{n+2}}{a[b+n+1]_{b+1}} \sum_{i=0}^n (-1)^i(a+i)[b+i-1]_{b-1}{b+n+1 \choose b+i+1} .\]
We may ignore the nonzero constant in front of the sum and write the rest as
\[ \sum_{i=-(b+1)}^n (-1)^i(a+i)[b+i-1]_{b-1}{b+n+1 \choose b+i+1} -  \sum_{i=-(b+1)}^{-1} (-1)^i(a+i)[b+i-1]_{b-1}{b+n+1 \choose b+i+1}.\]
Since $p(x)=(a+x)[b+x-1]_{b-1}$ is a polynomial of degree $b$, the first sum vanishes by Lemma~\ref{lem-binom} below. In the second sum, the terms $ [b+i-1]_{b-1}$ are zero unless $i=-b-1$ or $i=-b$. Thus, the sum is
\begin{gather*} -\big( (-1)^{-b-1} (a-b-1) [-2]_{b-1}{ n+b+1 \choose 0} + (-1)^{-b} (a-b) [-1]_{b-1} { n+b+1 \choose 1} \big) \\
= \pm [-1]_{b-1} (bn-a(n+1)).
\end{gather*}
Now the result follows. 
\end{proof}

\begin{lemma} \label{lem-binom}
 Let $n>0$ be an integer and $p(x)$ a polynomial of degree less than $n$. Then
\[ \sum_{i=0}^n (-1)^i p(i) {n\choose i} = 0.\]
\end{lemma}

\begin{proof}
 We have for $0\leq l<n$ 
\[ \sum_{i=0}^n (-1)^i i(i-1)\cdots (i-l+1) {n\choose i} = \partial_x^l (1-x)^n |_{x=1} = 0.\]
The polynomials $x(x-1)\cdots (x-l+1)$ for $l=0,\ldots, n-1$ span the space of all polynomials of degree less than $n$.
\end{proof}

\begin{lemma} \label{lem-deriv1}
 Let $n, b>0$ be integers. Then the derivative
\[ \tilde{D} = \sum_{i=0}^{n} (-1)^i \frac{(b+n-i-1)!}{(b-1)!} {n \choose i} \partial_y^i\]
applied to monomials $y^{b+j}$, $j=0,\ldots, n-1,$ vanishes at $y=1$. 
\end{lemma}

\begin{proof}
 We apply the derivative $\tilde{D}$ to the monomial $y^{b+j}$ and evaluate at $y=1$ to get
\[ \sum_{i=0}^n (-1)^i \frac{(b+n-i-1)!}{(b-1)!}{ n\choose i} [b+j]_i.\]
Using that $b+j-i+1\leq b+n-i-1$, we simplify
\begin{align*} 
\frac{(b+n-i-1)!}{(b-1)!} [b+j]_i &= (b+j)(b+j-1)\cdots(b+j-i+1) (b+n-i-1)\cdots (b+1)b \\
&= [b+j]_{j+1} [b+n-i-1]_{n-j-1}. 
\end{align*}
The polynomial $p(x) = [b+n-x-1]_{n-j-1}$ has degree $n-j-1<n$, hence
\[ [b+j]_{j+1} \sum_{i=0}^n (-1)^i p(i) {n\choose i} = 0\]
by the previous lemma.
\end{proof}

\begin{remark} \label{simple.case} In the case $n=1$, the proof of Theorem~\ref{thm-toric} can be simplified considerably. For this, we first transform the polytope $m\Delta$ by a reflection across the $y$-axis, shear transformation and translation to get to $\tilde{\Delta}$ that has its left vertex at $(-1,b)$, right vertex at $(W-1, 0)$ and the single lattice point in the second column from the right at $(W-2, 0)$. Now the derivative $\partial_x^{W-2} \partial_y$ vanishes on all monomials except the monomial corresponding to the left vertex.
\end{remark}


\section{Proof of Theorem~\ref{thm-proj}}

Recall that the weighted projective plane $\PP(a,b,c)$ is defined as $\Proj k[x,y,z]$, where the variables $x,y,z$ have degree $a,b,c$, respectively. A relation $(e,f,-g)$ defines a homogeneous polynomial $x^ey^f - z^g$ of degree $d=gc$.

There is a degree map $deg: \RR^3\to \RR$ that maps $(u,v,w) \mapsto au+bv+cw$. The toric variety $\PP(a,b,c)$ is then defined by the triangle $\Delta = \deg^{-1}(d) \cap \RR^3_{\geq 0}$ in the plane $deg^{-1}(d) \isom \RR^2$ and lattice $deg^{-1}(d)\cap \ZZ^3 \isom \ZZ^2$. With $d$ coming from the relation, we choose $(0,0,g)$ as the origin of the plane. The unit vector in the ``vertical'' direction is then  $(e,f,-g)$.

A divisor defined by a degree $d$ homogeneous polynomial in $k[x,y,z]$ has self-in\-ter\-sec\-tion number $d^2/(abc)$, which is the width of $\Delta$: 
\[ w = \frac{(gc)^2}{abc} = \frac{g^2 c}{ab}.\]
This identifies condition (1) of the theorem with condition (1) in Theorem~\ref{thm-toric}. To identify conditions (2) in the two theorems, we count lattice points in the columns of $m\Delta$.

Let us construct a linear function $h$ on $\RR^3$ that takes value $i$ on the column with index $i$ in $m\Delta$. Since $(0,0,g)$ and $(e,f,0)$ lie in column $0$, the function $h$ must be the dot product with 
\[  \alpha(f,-e,0) \]
for some constant $\alpha$. We can use $h$ to compute $w$. The two nonzero vertices of $\Delta$ are
\[ (\frac{cg}{a},0,0), \quad (0,\frac{cg}{b},0).\]
Thus, 
\[ w = \alpha( \frac{cg}{a}f - \frac{cg}{b}(-e) ) = \alpha \frac{(cg)^2}{ab},\]
from which we solve $\alpha = 1/c$. (Note that we chose  the vertex $(\frac{cg}{a},0,0)$ to be on the right hand side of the plane and the vertex $(0,\frac{cg}{b},0)$  on the left hand side.)

Consider $m\Delta$ and its vertex (on the left hand side)  $P=(0,\frac{mcg}{b},0)$. Instead of counting lattice points $Q$ in the second column from the left, we count lattice points $Q-P \in \ker(deg)\cap\ZZ^3$, such that $h(Q-P) = 1$.  These points are of the form $(u,v,w) \in \ZZ^3$, $u,w\geq 0$, $v\leq 0$, satisfying the equations 
\begin{gather*}
   h(u,v,w)=1 \quad \Leftrightarrow \quad  \frac{f}{c}u - \frac{e}{c}v = 1\\
  deg(u,v,w)=0  \quad \Leftrightarrow \quad au + bv+ cw = 0.
  \end{gather*}
There is a possibly non-integral point
\[ \frac{1}{g}(b,-a,0) \]
satisfying these equations. Any other point is obtained from this one by subtracting a rational multiple of $(e,f,-g)$:
\[ (u,v,w) =  \frac{1}{g}(b,-a,0) + \frac{\delta}{g}(e,f,-g),\quad \delta\leq 0.\]
Changing $v$ to $-v$, we get that the number $n$ of lattice points in the second column equals the number of integers $\delta\leq 0$, such that 
\[ (b,a) + \delta(e,-f) \]
has both components non-negative, divisible by $g$.

By a similar argument we get that the number of lattice points in the $n$-th column from the right is the number of integers $\gamma\geq 0$, such that 
\[ (n-1)(b,a) + \gamma(e,-f) \]
has both components non-negative and divisible by $g$. 

Finally, if the $(n+1)$-th column from the right has a lattice point on the top edge, then this point corresponds to the solution $\epsilon=0$, such that
\[ n(b,a) + \epsilon(e,-f)\]
has both components nonnegative and divisible by $g$. This happens if and only if $n(b, a) = (0,0) \pmod{g}$.\qed


\section{The moduli space $\overline{M}_{0,n}$}

We show that the Cox ring of $\overline{M}_{0,n}$ is not finitely generated if the characteristic of $k$ is $0$ and $n \geq 13$. For this we use the method of Castravet and Tevelev \cite[Proposition 3.1]{CastravetTevelev} to reduce to the case of a weighted projective plane blown up at the identity $t_0$ of its torus.

Recall that the moduli space $\overline{M}_{0,n}$ of stable $n$-pointed genus zero curves has been described by Kapranov as the iterated blowup of $\PP^{n-3}$ along proper transforms of linear subspaces spanned by $n-1$ points in linearly general position. The Losev-Manin moduli space $\overline{L}_n$ is constructed similarly by blowing up $\PP^{n-3}$ along proper transforms of linear subspaces spanned by $n-2$ points in linearly general position.
The space $\overline{L}_n$ is a toric variety and its fan $\Sigma_n$ is the barycentric subdivision of the fan of $\PP^{n-3}$. More precisely, the fan $\Sigma_n$ has rays generated by all vectors in $\RR^{n-3}$ such that each entry is equal to either $0$ or $1$, and all their negatives.

The main reduction step follows from the result that, given a surjective morphism $X \to Y$ of normal $\QQ$-factorial projective varieties, if $X$ has a finitely generated Cox ring, then so does $Y$ (see Okawa \cite{Okawa}). 
Using \cite[Proposition 3.1]{CastravetTevelev} or its corollary Theorem~\ref{thm-CT}, for suitable values of $n,a,b,c$ (e.g. $n=13$ and $(a,b,c) = (26,15,7)$, see \ref{example.configuration}), one can construct a rational map $\Bl_{t_0} \overline{L}_{n} \dra \Bl_{t_0} \PP(a,b,c)$ as a sequence of such surjective morphisms and small modifications (i.e. isomorphisms in codimension $1$) of normal $\QQ$-factorial projective varieties. Moreover, there exist surjective morphisms $\overline{M}_{0,n} \rightarrow \Bl_{t_0} \overline{L}_n$ and $\overline{M}_{0,n+1} \rightarrow \overline{M}_{0,n}$. Small modifications do not change the Cox ring, hence the non-finite generation of a Cox ring of $\overline{M}_{0,n}$ would follow from the non-finite generation of a Cox ring of $\Bl_{t_0} \PP(a,b,c)$.

The following is an immediate corollary of the main reduction result of Castravet and Tevelev \cite[Proposition 3.1]{CastravetTevelev}.

\begin{theorem}[Castravet-Tevelev \cite{CastravetTevelev}]  \label{thm-CT}
Let $\overline{L}_n$ be defined by the fan $\Sigma_n$ in the lattice $N=\ZZ^{n-3}$, as above. Suppose there exists a saturated sublattice $N'\subset N$ of rank $n-5$, such that
\begin{enumerate}
\item The vector space $N' \otimes \QQ$ is generated by rays of $\Sigma_n$.
\item There exist three rays of $\Sigma_n$ with primitive generators $u,v,w$ whose images generate $N/N'$ and such that $au+bv+cw=0 \pmod{N'}$ for some integers $a,b,c >0$, with $\gcd(a,b,c)=1$.
\end{enumerate} 
Then there exists a rational map $\Bl_{t_0} \overline{L}_n \dra \Bl_{t_0} \PP(a,b,c)$ that is a composition of rational maps each of which is either a small modification between normal $\QQ$-factorial projective varieties or a surjective morphism between normal $\QQ$-factorial projective varieties.
In particular, if $\Bl_{t_0} \PP(a,b,c)$ does not have a finitely generated Cox ring, then $\Bl_{t_0} \overline{L}_n$ (and $\overline{M}_{0,n}$) does not have it either.
\end{theorem}

\subsection{Proof of Theorem~\ref{thm-moduli}}  \label{example.configuration}
We show that Theorem~\ref{thm-CT} applies in the case $n=13$ and $(a,b,c) = (26,15,7)$.

Let $e_1,\ldots,e_{10}$ be the canonical basis of $\ZZ^{10}$. Let 
\begin{align*}
a_1&=e_1+e_5, &
a_6&=e_1+e_2+e_3+e_4+e_{10},\\
 a_2&=e_1+e_2+e_6, &
 a_7&=e_5+e_6+e_7+e_8+e_9+e_{10},\\
  a_3&=e_1+e_2+e_3+e_7, & a_8&=e_4+e_5+e_7,
  \\
   a_4&=e_1+e_2+e_3+e_4+e_8, &
    a_9&=e_1,\\
    a_5&=e_1+e_2+e_3+e_4+e_9, &
    a_{10}&=e_4.
\end{align*}
The matrix $A$ with columns $a_1,\ldots,a_{10}$ has determinant $1$, so $a_1,\ldots,a_{10}$ form a basis of $\ZZ^{10}$. Let $u=e_1$, $v=e_2$ and $w=-4u-2v+2a_1+a_2+a_3-a_8+a_{10}=e_3+e_5+e_6$. So, we have that 
\begin{align*}
26u+15v+7w 		& = 11a_1+8a_2+4a_3+a_4+a_5+a_6-a_7-3a_8	                          \\
         a_9 	&= u                                 	                              \\
	       a_{10}	&= 4u+2v+w-2a_1-a_2-a_3+a_8 
\end{align*}
Let  $N'\subset \ZZ^{10}$ be the sublattice generated by $a_1,\ldots, a_8$. Then $\ZZ^{10}/N'$ is generated by $a_9,a_{10}$, both of which can be expressed in terms of $u,v,w$. The vectors $u,v,w$ satisfy the relation 
\[26u+15v+7w = 0 \pmod{N'}.\]

We have seen that $\Bl_{t_0} \PP(26,15,7)$ does not have a finitely generated Cox ring (see Example~\ref{example.wps}), hence by Theorem~\ref{thm-CT}, for any $n \geq 13$ the moduli space $\overline{M}_{0,n}$ does not have a finitely generated Cox ring either.

\bibliographystyle{plain}
\bibliography{cox}

\end{document}